\documentclass[a4paper]{amsart}

\usepackage[dvipdfm]{hyperref}

\usepackage{mathptmx}
\usepackage{amsmath}
\usepackage{amssymb}
\usepackage{amsthm}
\usepackage{array}
\usepackage[all,tips,line]{xy}
\usepackage{enumerate}
\usepackage{graphicx}
\usepackage{lmodern}
\usepackage{mathrsfs}

\hypersetup{colorlinks,citecolor=blue,linkcolor=blue}

\newtheorem{theorem}{Theorem}[section]
\newtheorem{lemma}[theorem]{Lemma}
\newtheorem{prop}[theorem]{Proposition}
\newtheorem{cor}[theorem]{Corollary}

\theoremstyle{definition}
\newtheorem{definition}[theorem]{Definition}

\theoremstyle{remark}
\newtheorem{rmk}[theorem]{Remark}

\numberwithin{equation}{section}

\DeclareMathOperator{\GL}{GL}
\DeclareMathOperator{\Ort}{O}

\DeclareMathOperator{\SO}{SO}
\DeclareMathOperator{\PO}{PO}

\DeclareMathOperator{\rk}{rank}

\DeclareMathOperator{\vol}{vol}

\newcommand{\disc}{\mathrm{disc}}

\newcommand{\Z}{\mathbb Z}
\newcommand{\Q}{\mathbb Q}
\newcommand{\R}{\mathbb R}
\newcommand{\C}{\mathbb C}

\renewcommand{\O}{\mathcal O}

\newcommand{\Hy}{\mathbf H}

\newcommand{\Hct}{H_{\mathrm{ct}}}

\renewcommand{\P}{\mathbf P}
\newcommand{\G}{\mathbf G}

\newcommand{\sG}{\widetilde{\mathbf{G}}}
\newcommand{\sC}{C}

\newcommand{\ZZ}{\mathbf Z}

\newcommand{\UU}{\mathbf{U}}

\newcommand{\CC}{\mathcal C}
\newcommand{\Bl}{\mathcal B}

\newcommand{\bs}{\backslash}

\newcommand{\Ad}{\mathrm{Ad}}
\newcommand{\g}{\mathfrak{g}}

\begin{document}
\title[On volumes of quasi-arithmetic hyperbolic lattices]{On volumes of quasi-arithmetic \\ hyperbolic lattices}

\begin{abstract} We prove that the covolume of any quasi-arithmetic hyperbolic
        lattice (a notion that generalizes the definition of arithmetic
        subgroups) is a rational multiple of the covolume of an arithmetic
        subgroup. As a corollary, we obtain a good description for the shape of
        the volumes of most  of the known hyperbolic $n$-manifolds with $n >3$.
\end{abstract}

\author{Vincent Emery}
\thanks{This work is supported by the Swiss National Science Foundation, Project number PP00P2\_157583}

\address{
Universit\"at Bern\\
Mathematisches Institut\\
Sidlerstrasse 5\\
CH-3012 Bern\\
Switzerland
}
\email{vincent.emery@math.ch}

\date{\today}


\maketitle

\section{Introduction}
\label{sec:intro}

\subsection{}

Let $\Hy^n$ be the hyperbolic $n$-space, with group of isometries $G =
\PO(n,1)$.  Let $k$ be a number field with ring of integers $\O_k$.  An
absolutely simple
adjoint algebraic $k$-group $\G$ will be called {\em admissible} (for $G$)  if
$\G(k\otimes_\Q \R) \cong G \times K$, where $K$ is compact (possibly trivial).
Assuming $n \ge 4$ this forces $k$ to be totally real, and we can fix an
inclusion  $k \subset \R$ such that $G = \G(\R)$.  By the theorem of Borel and
Harish-Chandra, any subgroup $\Gamma_0 \subset \G(\R)$ commensurable with
$\G(\O_k)$ is a lattice in $G$. Such a subgroup is called \emph{arithmetic}.
Since we assume that $\G$ is adjoint, we have necessarily $\Gamma_0 \subset
\G(k)$; see \cite[Proposition 1.2]{BorPra89}. 

Following Vinberg \cite{Vinb67} we call \emph{quasi-arithmetic} a lattice of $G$ 
that is obtained as a subgroup of $\G(k)$ for $\G$ admissible. We call
``properly quasi-arithmetic'' such a lattice if it is not arithmetic. First examples 
were obtained by Vinberg, who considered reflection groups \cite{Vinb67}.
The construction of Belolipetsky and Thomson \cite{BelThom11} proves the
existence of infinitely many commensurability classes of  
properly quasi-arithmetic hyperbolic lattices in any dimension $n>2$; see
Thomson \cite{Thoms}.

\begin{rmk}
        \label{rmk:defns-coincide}
        Suppose that $\Gamma' \subset \G(\R)$ is commensurable with the quasi-arithmetic
        lattice $\Gamma \subset \G(k)$. Then it follows from the work of Vinberg
        \cite{Vinb71} that
        $\Gamma' \subset \G(k)$ (using the fact that $\G$ is adjoint). In particular, this shows that
        our definition of quasi-arithmeticity coincides with the one from  \cite{Vinb67} and \cite{Thoms}.
\end{rmk}

\begin{rmk}
        \label{rmk:G-k-always-possible}
        It follows from Weil's local rigidity theorem that any lattice $\Gamma
        \subset G$ can be embedded in $\G(k)$ for $k$ some number field and $\G$
        some $k$-group such that $\G(\R) = G$ (see \cite[Ch.~1 Sect.~6.2]{OniVinb-I}). 
        However, in general $\G$ does not have to be admissible.
\end{rmk}

\subsection{}
It is clear from the definition that the covolume of an arithmetic  
subgroup $\Gamma_0 \subset \G(k)$ is commensurable with the
covolume of $\G(\O_k)$ (as lattices in $\G(\R)$). In this paper we prove 
that this holds for any quasi-arithmetic lattice $\Gamma \subset \G(k)$.

\begin{theorem}
      \label{thm:vol-prop}
      Let $\Gamma \subset \G(k)$ be a quasi-arithmetic lattice.
      Then the covolume of $\Gamma$ is a rational multiple of 
      the covolume of $\G(\O_k)$.
\end{theorem}


For $n$ even this is an obvious consequence of the generalized Gauss-Bonnet
formula. However, we obtain the result as a consequence of Theorem
\ref{thm:fund-class} below, which is of interest for even dimensions as well.
In this stronger form our theorem also has application to the study of
arithmetic lattices (see Corollary \ref{cor:integ-mult}).
Moreover -- and despite the seemingly particular nature of properly
quasi-arithmetic lattices -- we will show how this notion permits to better
understand both arithmetic lattices (Corollary \ref{cor:main-thm})
and non-quasi-arithmetic lattices (Section \ref{sec:list-maniflds}).

\begin{rmk} \label{rmk:dim-3} 

        For $n = 3$ the result stated in Theorem \ref{thm:vol-prop} appears as
        a particular consequence of the known theory about the Bloch invariant;
        see \cite{NeuYan99}  and \cite[Section~12.7]{MaclReid03}. 
        Very briefly,
        for a hyperbolic $3$-manifold $M$ its {\em Bloch invariant} $\beta(M)$ takes value
        in the {\em Bloch group} $\Bl(k) \otimes \Q$, which has dimension $1$
        when $\G$ is admissible (in this case the {\em invariant trace field}
        $k$ must have exactly one complex place).  The result then follows by
        applying the Borel regulator on $\beta(M)$ (which gives the volume).
        The theory has been generalized for higher dimensions
        \cite[Section~8]{NeuYan99}, however with $\beta(M)$ taking values in
        higher (pre-)Bloch groups of $\C$ (or $\overline{\Q}$) -- instead of
        $k$.
        This does not permit direct volume comparisons, the vector spaces involved
        having infinite dimensions.
        See also Goncharov \cite{Gonch99}, who defines a
        similar invariant in the $K$-theory groups $K_n(\overline{\Q}) \otimes
        \Q$.
\end{rmk}

The covolume of $\G(\O_k)$ can be computed up to a rational
in terms of invariants of $\G$ and $k$ (see Ono \cite{Ono66},  and Prasad \cite{Pra89}).
We discuss in Section \ref{sec:example} the
particular case of nonuniform lattices. We also include there numerical
comparisons for two properly quasi-arithmetic lattices 
-- two reflection groups in dimensions $5$ and $7$ --
that illustrates Theorem \ref{thm:vol-prop}. 

\subsection{}

As a corollary of Theorem \ref{thm:vol-prop} we obtain the following
result.  We do not know if it can be obtained by a more direct proof --
even in the arithmetic case. 

\begin{cor}
       Let $M = \Gamma \bs \Hy^n$ be an orientable hyperbolic manifold with
       $\Gamma \subset \G(k)$ quasi-arithmetic. Suppose that $M$ contains a 
       totally geodesic (connected) separating hypersurface $S$ of finite volume,
       and let $M^+ \subset M$ be one of the two parts delimited by $S$. 
       We denote by $H \subset \Hy^n$ the hyperplane that covers $S$, and we suppose
       that  the reflection through $H$ belongs to $\G(k)$.  
       Then $\vol(M^+)$ is a rational multiple of $\vol(M)$.
       \label{cor:main-thm}
\end{cor}



\begin{proof}
Let us write $\Gamma^+ \subset \Gamma$ for the subgroup corresponding to the
fundamental group of $M^+$.
We consider the manifold $V$ that is obtained by gluing two copies of $M^+$
along the boundary $S$. We have that $V$ is a complete hyperbolic manifold
that can be written as  $V = \Lambda \bs \Hy^n$, where $\Lambda \subset \G(\R)$
is the subgroup generated by $\Gamma^+ \cup g \Gamma^+ g^{-1}$ and 
$g$ is the reflection through $H$.
Obviously $V$ has finite volume, so that $\Lambda \subset \G(k)$ is a
quasi-arithmetic lattice.  By Theorem \ref{thm:vol-prop} we conclude that $\vol(V) =
2 \vol(M^+)$ is a rational multiple of $\vol(M)$.  
\end{proof}

\subsection{}
\label{sec:list-maniflds}

The following list covers all  the {\em currently known}
hyperbolic lattices  for $n > 3$ (up to commensurability). We indicate
the relevant information about their volumes (for $n$ odd):
\begin{enumerate}
        \item Arithmetic lattices, for which the volumes are precisely
                computed up to rationals (see \cite{Ono66,Pra89} and Section
                \ref{sec:example}). 
        \item Quasi-arithmetic lattices; Theorem \ref{thm:vol-prop} shows that
                up to commensurability their covolumes are the same as in (1).
       \item Lattices that come from the {\em interbreeding} constructions 
               (see Gromov and Piateski-Shapiro \cite{GroPS87},
               and generalizations \cite{Raim12,GelLev14}); note that those are
               not quasi-arithmetic \cite[Theorem 1.6]{Thoms}. 
               From their construction and  Corollary \ref{cor:main-thm}
               we obtain that their volumes are rational linear combinations 
               of volumes from (1), in any case (the result being already clear
               when the construction only involves nonseparating hypersurfaces). 
        \item Non-quasi-arithmetic reflection groups; we do not have any
                information about what shape their volumes can take in general.
                Note however that some of these groups are obtained by the
                interbreeding construction; see for instance Vinberg
                \cite{Vinb15}.

\end{enumerate}

\begin{rmk}
For $n = 3$ hyperbolic manifolds can be obtained by performing Dehn filling on
link complements, and this possibly provides lattices that are not of any of
the types listed above. 
\end{rmk}

\subsection{}
\label{sec:strategy-unif}

The idea of the proof of Theorem \ref{thm:vol-prop} is easily summarized in the case when $\Gamma$ is
uniform. Up to passing to a subgroup of finite index, we may assume that
$\Gamma$ is torsion-free and contained in the connected component $G^\circ$. 
Then $M = \Gamma \bs \Hy^n$ is a compact orientable manifold, and this  
provides a ``fundamental class'' $[\Gamma] \in H_n(\Gamma) \cong H_n(M) \cong \Z$ (which
corresponds to the generator with the same orientation as $\Hy^n$).
Denoting by $j: \Gamma \to G$ the inclusion map,  
the result follows immediately from the two following observations:
\begin{enumerate}
        \item the induced map $j_*$ on group homology factors as
                $H_n(\Gamma) \to H_n(\G(k)) \to H_n(G)$, 
                and the middle term is known to have rank one; see Proposition
                \ref{prop:Borel-rk-1}.
        \item there is a linear map $v: H_n(G) \to \R$ (independent of $\Gamma$) 
                such that $v(j_*([\Gamma])) = \vol(M)$; see Section \ref{sec:fund-class}.
\end{enumerate}

\subsection{}
\label{sec:intro-relative-homol}

The main effort of this article is to present a proof that includes the more
difficult case of  nonuniform lattices.
We first need to recall a convenient way to consider the fundamental 
class in the general case. We refer to Section \ref{sec:fund-class} for details. 
Let $\Omega = \partial \Hy^n$ be the geometric boundary of
$\Hy^n$, endowed with the $G$-action. For a subgroup $S \subset G$, we define
its homology \emph{relative to} $\Omega$ by 
\begin{align}
        H_n(S, \Omega) &= H_{n-1}(S, J\Omega),
        \label{relat-homology}
\end{align}
where $J\Omega$ is the kernel of the augmentation map $\Z \Omega \to \Z$ (in
particular we can see $J\Omega$ as an $S$-module).
For any torsion-free lattice $\Gamma \subset G^\circ$ we have $H_n(\Gamma,
\Omega) \cong \Z$, and we denote by $[\Gamma] \in H_n(\Gamma, \Omega)$ 
the generator corresponding to the positive orientation 
(see Definition \ref{def:fund-class} and Remark \ref{rmk:CC-Omega-orient}).
For $\Gamma$ uniform
there is a canonical isomorphism $H_n(\Gamma, \Omega) \cong H_n(\Gamma)$ and we
recover the usual notion of the fundamental class.

For any lattice $\Gamma \subset G$ we consider the map $j_*: H_n(\Gamma, \Omega)
\to H_n(G, \Omega)$ on the relative homology, induced by the inclusion $j:
\Gamma \to G$.  The following result implicitly appears  
in Neumann and Yang \cite[Sect.~3--4]{NeuYan99}. We discuss the proof
in Section \ref{sec:fund-class}.

\begin{prop}
        \label{prop:class-determ-vol}
        There exists a linear map $v: H_n(G, \Omega) \to \R$ such that for any
        torsion-free lattice $\Gamma \subset G^\circ$ one has $v(j_*([\Gamma])) =
        \vol(\Gamma \bs \Hy^n)$.
\end{prop}

\subsection{}

We will write $\G(k)^+$ for the intersection $\G(k) \cap G^\circ$. After passing to a
finite index torsion-free subgroup, Theorem \ref{thm:vol-prop} is a direct consequence of
the following result, together with Proposition \ref{prop:class-determ-vol}. 

\begin{theorem}
        Let $\G$ be an admissible $k$-group. There exists a rank one
        $\Z$-sub\-module  $L \subset H_n(G, \Omega)$ such that $j_*([\Gamma]) \in L$ for
        any torsion-free quasi-arithmetic lattice $\Gamma \subset \G(k)^+$.
        \label{thm:fund-class} 
\end{theorem}

The proof of Theorem \ref{thm:fund-class} is established in Sections
\ref{sec:struct-cusps}--\ref{sec:conclusion}.
In view of Remark \ref{rmk:dim-3}, we will assume $n > 3$ in the proofs. This
permits to use a uniform notation (the main difference for $n=3$ is that the
field of definition $k$ is not totally real). 

\subsection{}
\label{sec:vol-integral-mult}
One feature of our approach is that the fundamental classes 
are compared in a
$\Z$-module, namely $L \subset H_n(G, \Omega)$. This contrasts with the
Bloch invariant
approach, which considers $\Q$-vector spaces (cf.\ Remark \ref{rmk:dim-3}).
To illustrate the advantage in doing so, we note the following simple
corollary  concerning the distribution of covolumes within a commensurability class
of arithmetic subgroups. Recall that such a class contains infinitely many
maximal subgroups (see \cite[Prop.~1.4~(iv)]{BorPra89}). 

\begin{cor}
        \label{cor:integ-mult}
       Let $\G$ be an admissible $k$-group.
       There exists a number $c > 0$ such that for any arithmetic subgroup
       $\Gamma \subset \G(k)$ we have that $\vol(\Gamma\bs\Hy^n)$ is an
       integral multiple of $c$.
\end{cor}
\begin{proof}
   
        Since $\G$ has trivial center, we have that $\G \cong \Ad\, \G$ can be
        identified as a subgroup of $\GL(\g)$, where $\g$ denotes the Lie
        algebra (defined over $k$) of $\G$. Using Weil's restriction of scalars,
        we can then identify $\G(k)$ with the rational points of an algebraic
        $\Q$-subgroup $H \subset \GL(\g_0)$, where $\g_0$ is a $\Q$-form of
        $\g$. Under this identification, any arithmetic subgroup $\Gamma \subset
        \G(k)$ corresponds to an arithmetic subgroup (defined over $\Q)$ of
        $H(\Q)$. In particular (see \cite[Ch.~1 Prop.\ 7.2]{OniVinb-I}), $\Gamma$ stabilizes a $\Z$-lattice $L \subset
        \g_0$, and choosing a basis of $L$ we can embed $\Gamma \subset
        \GL_m(\Z)$. For a fixed field $k$, this integer $m$ is independent of
        the arithmetic subgroup $\Gamma \subset \G(k)$.

       From Minkowski's lemma we thus have an upper bound $A > 0$
       such that any arithmetic lattice $\Gamma \subset \G(k)$ contains a
       torsion-free subgroup $\Gamma_0 \subset \Gamma$ of index $[\Gamma:
       \Gamma_0] \le A$. Moreover, by doubling $A$ 
       we ensure that $\Gamma_0 \subset \G(k)^+$. By Theorem \ref{thm:fund-class}
       and Proposition \ref{prop:class-determ-vol}
       there exists $c' > 0$ such that for all these subgroups $\Gamma_0$ we have 
       $\vol(\Gamma_0\bs\Hy^n) \in  c'\Z$. The result follows by choosing for $c$
       the number $c'$ divided by the lowest multiple common to all integers
       $\le A$.
\end{proof}

In dimension $n=3$ this result was obtained by Borel in \cite{Bor81},
as a consequence of his volume formula; it answered positively a question
of Thurston \cite[6.7.6]{Thur80} (the case of nonarithmetic lattices
being solved by the commensurator theorem of Margulis). 
The work of Borel and Prasad  
\cite{BorPra89}, which relies on the volume formula \cite{Pra89}, provides
a lot of information about the volume distribution
for arithmetic lattices in very generic situations --
including the case of $\PO(n,1)$.
However it does not seem that Corollary \ref{cor:integ-mult} could be easily obtained
from their results.

\begin{rmk}
        Our result also shows that all covolumes of the quasi-arithmetic
        \emph{torsion-free} lattices $\Gamma \subset \G(k)$ are integral multiples of a
        single number. However it is not clear if this holds true for lattices containing
        torsion.
\end{rmk}

\subsection*{Acknowledgement} I would like to thank Thilo Kuessner for his help concerning
Section 3, and Pavel Tumarkin for pointing out that the work of his student Mike Roberts
contains new quasi-arithmetic reflection groups. I thank Steve Tschantz for the numerical
computations that permit the volume comparisons in Section \ref{sec:example}; this part of
the work has been realized in the framework of the AIM Square program ``Hyperbolic
geometry beyond dimension three''. We thank the American Institute of Mathematics for
their support. I also thank Inkang Kim and  Gopal Prasad for pointing out some mistakes in
earlier versions, and  Scott Thomson, Jean Raimbault, and the referee for helpful comments.

\section{Volume computations for the nonuniform case}
\label{sec:example}

\subsection{}
\label{sec:example-disc}
Let $\G$ be an admissible $k$-group for $G = \PO(n,1)$, and suppose that 
$\Gamma \subset \G(k)$ is a nonuniform lattice. Then $\Gamma$ must contain some nontrivial
unipotent elements, which means that $\G$ is isotropic. For $n > 3$ this is only possible if $k = \Q$,
and $\G$ corresponds to the adjoint group of $\SO(f)$ for $f$ a quadratic form over $\Q$ (cf.\
\cite[Lemma 2.2]{LiMil93}).

From now on suppose that $n$ is odd, with $n = 2m-1$. Let us define
\begin{align} \delta  = (-1)^m \disc(f), \label{eq:delta} \end{align} where
$\disc(f) \in \Q^\times/(\Q^\times)^2$ is the discriminant of $f$, and consider
the field $\ell = \Q(\sqrt{\delta})$. Let $D_\ell$ be its
discriminant, and $\zeta_\ell$ its Dedekind zeta function.

\subsection{} 
\label{sec:example-vol}

For a certain natural normalization of the Haar measure on a semisimple Lie
group $G$, the 
covolume of any arithmetic subgroup of $G$ can be obtained up to a rational
by using Prasad's volume formula \cite{Pra89} (the formula also permits precise
computations in many cases). Its application in the case $G = \PO(n,1)$ ($n$
odd) is worked out for instance in \cite[Sect.\ 2.6--7]{BelEme}.
The difference between Prasad's normalization of
the measure and the hyperbolic volume on $\Hy^n$ is explained in 
\cite[Sect.\ 2.1]{BelEme}. Together with Theorem \ref{thm:vol-prop} 
one can deduce the value $\vol(\Gamma\bs\Hy^n)$
up to a rational for any quasi-arithmetic lattice $\Gamma \subset G$. 
We give in the following proposition the values for the case  $k =\Q$.

\begin{prop}
        \label{prop:volumes-nonunif}
        Let $\Gamma \subset \G(\Q)$ be a nonuniform quasi-arithmetic lattice
        of $\PO(n,1)$ with $n \ge 5$ odd,
        and let $\ell/\Q$ be the field extension defined above in Section
        \ref{sec:example-disc}. 
        \begin{enumerate}
                \item If $\ell = \Q$, then  the covolume of $\Gamma$ is a rational
                        multiple of the Riemann zeta function
                        evaluated at $m = \frac{n+1}{2}$: 
                        \begin{align*}
                                \vol(\Gamma\bs \Hy^n)\; &\in\; \zeta(m) \cdot \Q^\times;
                        \end{align*}

                \item otherwise we have 
                        \begin{align*} 
                                \vol(\Gamma\bs \Hy^n) \; &\in \;
                                |D_\ell|^{n/2} \cdot
                                \frac{\zeta_\ell(m)}{\zeta(m)} \cdot\Q^\times.
                        \end{align*} 
\end{enumerate} 
\end{prop}

\begin{rmk}
       The quotient $\zeta_\ell/\zeta$ might alternatively be described as a
       Dirichlet $L$-function. 
\end{rmk}

\begin{rmk}
        Prasad's formula also provides similar volume formulas in the compact
        case, i.e., for $k \neq \Q$. In this case $\zeta$ is to be replaced by
        $\zeta_k$, and $\ell$ is a quadratic extension of $k$ (except for the
        special case of ``triality forms''). Both
        discriminants $D_\ell$ and $D_k$ then appear in the formula.
\end{rmk}

\subsection{}
\label{sec:example-dim-5}
We consider the $5$-dimensional hyperbolic polytope $P_5 \subset \Hy^5$
that corresponds to the Coxeter diagram 
given in \eqref{eq:polytope} (see \cite[Sect.~4]{Vinb67} for the notation).
Let us denote by
$\Delta_5 \subset \PO(5,1)$ the discrete subgroup generated by the 
reflections through the hyperplanes delimiting $P_5$.
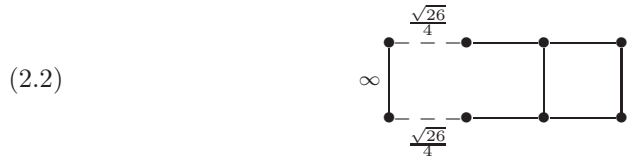
\begin{align}
        \label{eq:polytope}
        \begin{split}
       \xymatrix{
               *{\bullet} \ar@{--}[r]^{\frac{\sqrt{26}}{4}} 
               \ar@{-}[d]_\infty
               & *{\bullet} \ar@{-}[r] 
               & *{\bullet} \ar@{-}[r] \ar@{-}[d] 
               & *{\bullet} \ar@{-}[d] \\
               *{\bullet} \ar@{--}[r]_{\frac{\sqrt{26}}{4}} 
               & *{\bullet} \ar@{-}[r]  
               & *{\bullet} \ar@{-}[r] 
               & *{\bullet} 
       }
        \end{split}
\end{align}
This polytope appears in the list obtained by Mike Roberts in \cite{Roberts},
which contains many new examples of hyperbolic Coxeter polytopes of finite volume. 
The finiteness of $\vol(P_5)$ implies that $\Delta_5$ is a lattice -- nonuniform
since $P_5$ is noncompact.
Using a geometric integration Steve Tschantz has computed the following
numerical approximation for the volume of the polytope $P_5$: 
\begin{align}
       \vol(P_5) &\approx 0.0241330687945822699990.
        \label{eq:vol-Steve}
\end{align}

\subsection{}
\label{sec:dim-5-qa}
The Gram matrix of $P_5$ can be immediately deduced from the diagram
\eqref{eq:polytope}. From this matrix one can obtain (with the procedure used in the proof
of \cite[Theorem~2]{Vinb67}) a basis $\left\{ u_i \right\}$ of the Minkowski space $\R^{5,1}$ 
for which $(u_i, u_j) \in \Q$ and such that each $u_i$ is orthogonal to a hyperplane
delimiting $P_5$. With this one easily checks that $\Delta_5$ embeds as a subgroup of $\Ort(f, \Q)$ for
$f$ some quadratic form of discriminant $\disc(f) = -13 (\Q^\times)^2$.
In particular $\Delta_5$ is quasi-arithmetic; indeed, {\em properly}
quasi-arithmetic by applying \cite[Theorem~2]{Vinb67}. 
From Proposition \ref{prop:volumes-nonunif} we conclude that
$\vol(\Delta_5\bs\Hy^5) = \vol(P_5)$ is a rational multiple of
$13^{5/2} \zeta_\ell(3)/\zeta(3)$, where $\ell = \Q(\sqrt{13})$.
Numerical comparison  with
\eqref{eq:vol-Steve} (we use Pari/GP to evaluate the zeta functions)
then suggests the equality
\begin{align}
        \label{eq:volume-match}
        \vol(\Delta_5\bs\Hy^5) &= \frac{1}{23040} \cdot 13^{5/2} \cdot
        \frac{\zeta_\ell(3)}{\zeta(3)}.
\end{align}

\begin{rmk}
        The numerical match between \eqref{eq:volume-match} and
        \eqref{eq:vol-Steve}, together with the simplicity of the rational factor
        (note that $23040 = 2^9\cdot 3^2\cdot 5$), leave little doubt for the
        correctness of \eqref{eq:volume-match}. However, at this point we do not
        see any way to give a rigorous proof of this equality. 
\end{rmk}

\subsection{}
Let $P_7 \subset \Hy^7$ be the hyperbolic Coxeter polytope with the diagram: 
\begin{align}
        \label{eq:polytope-7}
        \begin{split}
       \xymatrix{
               *{\bullet} \ar@{--}[r]^{\frac{\sqrt{22}}{4}} 
               \ar@{-}[d]_\infty
               & *{\bullet} \ar@{-}[r] 
               & *{\bullet} \ar@{-}[r] 
               & *{\bullet} \ar@{-}[d] \ar@{-}[r]
               & *{\bullet} 
               \\
               *{\bullet} \ar@{--}[r]_{\frac{\sqrt{22}}{4}} 
               & *{\bullet} \ar@{-}[r]  
               & *{\bullet} \ar@{-}[r] 
               & *{\bullet} \ar@{-}[r] 
               & *{\bullet} 
       }
        \end{split}
\end{align}
As well as $P_5$ this polytope was found in \cite{Roberts}. We proceed as in Section
\ref{sec:dim-5-qa}: the corresponding reflection group 
$\Delta_7 \subset \PO(7,1)$ is properly quasi-arithmetic, defined as a
subgroup of an algebraic $\Q$-group determined by a quadratic form with
discriminant $-11$. Steve Tschantz has computed the following approximation:
\begin{align}
        \label{eq:volume-match-7}
        \vol(P_7) &\approx 0.000181338 \\
        \nonumber    &\approx \frac{11^{7/2}}{2^{13}\cdot 3^4 \cdot 5 \cdot 7}  
\cdot \frac{\zeta_\ell(4)}{\zeta(4)},
\end{align}
with $\ell = \Q(\sqrt{-11})$.
Again, this agrees with Proposition \ref{prop:volumes-nonunif}.

\section{Fundamental class and volume}
\label{sec:fund-class}

The purpose of this section is to prove Proposition \ref{prop:class-determ-vol},
mostly by following the approach of Neumann and Yang \cite[Sect.~3--4]{NeuYan99} (see also Kuessner \cite{Kuess12-TP}). 
In the following text $M = \Gamma\bs\Hy^n$ 
denotes a finite-volume orientable hyperbolic manifold. As
in Section \ref{sec:intro}, we denote by $\Omega$ the geometric boundary of $\Hy^n$.


\subsection{}
A point $x \in \Omega$ is a \emph{cusp} of
$\Gamma$ if it is the fixed point  of a parabolic element of $\Gamma$.  Let
$\CC \subset \Omega$ be the set of cusps of $\Gamma$. When $M$ is compact
then $\CC$ is empty.  
Let $Z$ be the end compactification of $M$, i.e., $Z$ is obtained by adjoining a
point $c_i$ to each of the (finitely many) cusps of $M$.
We consider a triangulation of $Z$, and we suppose (as we may) that each $c_i$ is a
vertex of this triangulation. This triangulation lifts to the covering space $X = \Hy^n \cup \CC$.
We denote by $C_\bullet(Z)$ the chain complex defined by the triangulation of $Z$, and by
$C_\bullet(X)$ the chain complex of its lift.
Then  $C_\bullet(X)$ is a $\Z\Gamma$-complex,
and $C_\bullet(X)_\Gamma = C_\bullet(Z)$. In particular,
\begin{align}
        H_n(C_\bullet(X)_\Gamma) = H_n(Z) \cong \Z.
        \label{eq:H_n-Z}
\end{align}

\subsection{}

For a free $\Z$-module with basis $S$, we denote by $JS$ the kernel of the
augmentation map $\Z S \to \Z$, which by definition sends any $x \in S$ to $1$.
Let $X_0 \subset X$ denotes the set of vertices of the lifted triangulation.
Then $X_0$ is $\Z$-basis of $C_0(X)$, and we have that $J X_0$ is
a $\Gamma$-submodule of $C_0(X)$.

\begin{lemma}
        The $\Gamma$-module $JX_0$ is isomorphic to  $J\CC \oplus F$, where
         $F$ is some free $\Gamma$-module.
         \label{lemma:JX-direct-sum}
\end{lemma}
\begin{proof}
        The group $\Gamma$ acts freely on $X_0 \cap \Hy^n$. In particular, if
        $\CC$ is empty then $JX_0$ is a free $\Z\Gamma$-module. When $\CC$ is
        not empty it suffices to take as $F$ the submodule generated by elements
        of the form $x - a$, with $x \in X_0 \cap \Hy^n$ and $a \in \CC$.
\end{proof}

By definition, the relative homology $H_n(\Gamma, \CC)$
is $H_{n-1}(\Gamma, J\CC)$. We then have the following.

\begin{prop}
       $H_n(\Gamma, \CC) = H_n(C_\bullet(X)_\Gamma)$. 
       \label{prop:CX-computes-rel-hom} 
\end{prop}

\begin{proof}
        Since $X$ is contractible and $\Gamma$ acts freely on $\Hy^n$, 
        we have that $C_{\bullet\ge 1}(X)$ is a free $\Z\Gamma$-resolution of $JX_0$. 
        Thus $H_{n-1}(\Gamma, JX_0) = H_n(C_\bullet(X)_\Gamma)$, and the former
        equals $H_{n-1}(\Gamma, J\CC)$ by Lemma \ref{lemma:JX-direct-sum}.
\end{proof}

The proposition, together with \eqref{eq:H_n-Z}, justifies the following.

\begin{definition}
        \label{def:fund-class}
        For $\Gamma$ as above, we define its \emph{fundamental class} $[\Gamma]$ 
        to be the generator of $H_n(\Gamma, \CC)$ that can by represented by 
        a sum of positively oriented simplices from $C_n(X)$.
\end{definition}

\begin{rmk}
        \label{rmk:CC-Omega-orient}
       The natural map $H_n(\Gamma, \CC) \to H_n(\Gamma, \Omega)$ is an
       isomorphism (see \cite[Lemma 2.2.5]{Kuess12-TP}).
       This justifies an option to take as an equivalent definition for $[\Gamma]$
       the (positively oriented) generator of $H_n(\Gamma, \Omega)$, as we did in Section
       \ref{sec:intro-relative-homol}. 
\end{rmk}

\subsection{}
\label{sec:tau}
We consider the free $\Z$-module $S_j(\Omega)$ that is 
generated by the $(j+1)$-tuples $(x_0, \dots, x_j)$ of distinct elements 
in $\Omega$ modulo the relations
\begin{align*}
        (x_0, \dots, x_j) = \mathrm{sgn}(\sigma) (x_{\sigma(0)}, \dots,
        x_{\sigma(j)}),
\end{align*}
for any permutation $\sigma$. Geometrically,
these generators correspond to the ideal geodesic simplices in $\Hy^n$
(with orientation).
With the standard boundary map, the chain complex $S_\bullet(\Omega)$ gives 
a resolution of $\Z$. Moreover, the isometry group $G$ acts on this complex, so that 
$S_{\bullet\ge 1}(\Omega)$ is a (non-free) $\Z G$-resolution of $J\Omega$. 

Let $\tau$ denote the inclusion $J\CC \to J \Omega$. Since 
$S_{\bullet\ge 1}(\Omega)$ is acyclic, for any free 
$\Z \Gamma$-resolution $D_\bullet \to J \CC$ the map $\tau$ extends uniquely up
to homotopy to a $\Z\Gamma$-chain complex map (see \cite[Lemma I.74]{Brown82}): 


\begin{align}
        \begin{split}
\xymatrix{
        D_{\bullet} \ar@{->}[d] \ar[r] & J\CC \ar[d]^\tau & \\
        S_{\bullet \ge 1}(\Omega) \ar[r]  & J\Omega 
}
        \end{split}
        \label{eq:tau}
\end{align}
In particular, this induces a canonical map $\tau_*: H_n(\Gamma, \CC) \to
H_n(S_\bullet(\Omega)_\Gamma)$. For $D_\bullet = C_{\bullet \ge 1}(X)$, 
the chain complex map may be explicitly given as follows. 
Take a set of $\Gamma$-representatives of points of $X_0 \cap \Hy^n$ and send
them to arbitrarily chosen distinct points in $\Omega \setminus \CC$.
Obviously such a choice determines uniquely a
$\Z\Gamma$-map $C_j(X) \to S_j(\Omega)$ for any $j \ge 1$. 


\subsection{}
Let $\nu: S_n(\Omega) \to \R$ be the linear map that assigns to any 
$n$-simplex its signed hyperbolic volume. Then $\nu$ is zero on boundary
elements: if $b = \partial c$ for some $c \in S_{n+1}(\Omega)$ then $\nu(b) =
0$.  Moreover, $\nu$ is obviously $G$-invariant, so that
for any subgroup $A \subset G$ we obtain an induced map
$\nu_*: H_n(S_\bullet(\Omega)_A) \to \R$. For the case $A = \Gamma$, we can
state the following (see \cite[end of the proof of Lemma 4.2]{NeuYan99}).

\begin{prop}
        \label{prop:v-alpha}
        We have $\nu_*(\tau_*([\Gamma])) = \vol(M)$. 
\end{prop}

\subsection{}
Consider $J \Omega$ as a $\Z G$-module, 
and let $\iota$ denote the identity $J\Omega \to J\Omega$.
Similarly as for $\tau$, it induces a canonical
map $\iota_*: H_n(G, \Omega) \to H_n(S_\bullet(\Omega)_G)$. Since $\tau$ agrees
with $\iota$ on its domain of definition, we have that the left square in the
following diagram is commutative. That the right square is also commutative is
obvious. Recall that $j : \Gamma \to G$ denotes the inclusion.

\begin{align}
        \begin{split}
\xymatrix{
        H_n(\Gamma, \CC) \ar[r]^-{\tau_*} \ar[d]_{j_*} & H_n(S_\bullet(\Omega)_\Gamma)
        \ar[d]\ar[r]^-{\nu_*} & \R \ar[d]^{id} & \\
        H_n(G, \Omega) \ar[r]^-{\iota_*} & H_n(S_\bullet(\Omega)_G)
        \ar[r]^-{\nu_*} & \R.
}
        \end{split}
        \label{eq:exact-sq-vol}
\end{align}

\begin{proof}[Proof of Proposition \ref{prop:class-determ-vol}]
        Set $v = \nu_* \circ \iota_*$.
        Then the result follows by combining Proposition \ref{prop:v-alpha} with the
        fact that the diagram \ref{eq:exact-sq-vol} is commutative.
\end{proof}

\section{Homology of Algebraic groups}
\label{sec:homol-alg-gps}

\subsection{}

We need to recall the following theorem, which appears in  
\cite[Prop.~XIII.3.9]{BorWal} for the anisotropic case, 
and in \cite[Theorem 2.1]{BorYang94} for $\sG$ isotropic.
Its proof combines the work of Borel, Garland, Yang, and uses the work of
Blasius-Franke-Grunewald \cite{BlFrGru94} in the isotropic case.
Here $\Hct^\bullet$ denotes the continuous cohomology. 
\begin{theorem}
  Let $\sG$ be a simply connected absolutely simple $k$-group.
  The natural map
  \begin{align}
    \label{eq:Borel-Yang}
    \Hct^\bullet(\sG(k\otimes_\Q \R), \R) \to H^\bullet(\sG(k), \R)  
  \end{align}
  is an isomorphism.
  \label{thm:Borel-Yang}
\end{theorem}

\subsection{}

Let $\widetilde{G} = \sG(k\otimes_\Q \R)$, for $\sG$ as above.
We have that $\widetilde{G}$ is connected. It is known that
\begin{align}
        \Hct^\bullet(\widetilde{G}, \R) &= H^\bullet(X_u, \R),
        \label{eq:H_ct-Xu}
\end{align}
where $X_u$ denotes the \emph{compact dual} symmetric space associated  
with $\widetilde{G}$; see Borel \cite[Sect.~10.2]{Bor74}.
We can now prove the following.

\begin{prop}
        Let $\G$ be an admissible $k$-group for $\PO(n,1)$. Then $H_n(\G(k))$
        has rank one. 
        \label{prop:Borel-rk-1}
\end{prop}
\begin{proof}

Let $\sG$ be the simply connected (algebraic) cover of $\G$,
and denote by $\pi$ the covering map $\sG \to \G$ and by $\sC$ the center of $\sG$.
From Galois cohomology we have an exact sequence 
\begin{align}
        1    \to \sG(k)/\sC(k) \stackrel{\pi} \to \G(k) \to A \to 1, 
  \label{eq:ex-seq-sG}
\end{align}
where $A$ is defined as the kernel of the map $H^1(k, \sC) \to H^1(k, \sG)$. 
The Lyndon-Hochschild-Serre spectral sequence (see \cite[Sect.~VII.6]{Brown82}) applied to
\eqref{eq:ex-seq-sG} takes the form
\begin{align}
        E^2_{pq} = H_p(A, H_q(\sG(k)/\sC(k), \R)) \Rightarrow H_{p+q}(\G(k), \R).
        \label{eq:LHS-1}
\end{align}
In all cases we have that $H^1(k, \sC)$ is torsion abelian of finite exponent. Thus $A$
may be written as a direct limit of finite groups, and it 
follows by exchanging homology and direct limit that 
$H_p(A, -)$ is zero  in \eqref{eq:LHS-1} unless $p = 0$. This shows that $H_n(\G(k), \R)$
has the same dimension as $H_n(\sG(k)/\sC(k), \R)$. Moreover, since $\sC$ is finite, the
spectral sequence 
\begin{align}
        E^2_{pq} = H_p(\sG(k)/\sC(k), H_q(\sC(k), \R)) \Rightarrow H_{p+q}(\sG(k), \R)
        \label{eq:LHS-2}
\end{align}
further shows that this dimension equals $\dim( H_n(\sG(k), \R))$. For $\G$ admissible we
have that the compact dual symmetric space $X_u$ of $\sG(k \otimes_\Q \R)$ has
the same dimension as $\Hy^n$
(more precisely, $X_u$ is the $n$-sphere). It then follows from \eqref{eq:Borel-Yang} and \eqref{eq:H_ct-Xu} 
that $H_n(\G(k))$ has rank one.
\end{proof}


\section{Algebraic structure at cusps}
\label{sec:struct-cusps}

\subsection{}
\label{sec:G_x}

Let $G = \PO(n,1)$, and take a point $x$ on the boundary  $\Omega$ of $\Hy^n$. 
Using the upper half space model with $x = \infty$ and the description of its isometry
group as conformal maps (cf.\ for instance \cite[Ch.~A]{BenPetr}) one sees 
that the stabilizer $G_x$ decomposes as a product
\begin{align}
       G_x = U \cdot A \cdot S, 
        \label{eq:G_x}
\end{align}
where 
\begin{itemize}
        \item $U \cong \R^{n-1}$ corresponds to the horospherical translations
                fixing $x$;
        \item $A \cong \R_{>0}$ corresponds to the homotheties centered at $0$;
        \item and $S \cong \Ort(n)$ is the rotation group around the axis $(0,
                x)$.
\end{itemize}

\subsection{}
\label{sec:Bieberbach}

Let $\G$ be an admissible $k$-group, so that $\G(\R) = G = \PO(n,1)$. We suppose
that $n > 3$. Let $\Gamma \subset \G(k)$ be a torsion-free lattice, and suppose that $x
\in \Omega$ is a cusp of $\Gamma$. In particular, $\G$ must be
isotropic with $k = \Q$ (cf. Section \ref{sec:example-disc}). The ``cusp subgroup'' 
$\Gamma_x$ acts discretely and cocompactly on $\R^{n-1}$. By Bieberbach theorem there
exists a normal finite index subgroup $\Gamma_x' \subset \Gamma_x$, that 
consists only of
horospherical translations. In the notation of Section \ref{sec:G_x}: $\Gamma_x' \subset U$, with
$\Gamma_x' \cong \Z^{n-1}$. We may assume that $\Gamma_x'$ is maximal abelian.
Then there
are only finitely many possibilities for the finite group $\Gamma_x/\Gamma_x'$
(see \cite[Ch.~4~Sect.~1.1]{OniVinb-I}). It follows that  there exists an integer $N$
such that $[\Gamma_x:\Gamma_x']$ divides $N$ for every cusp $x$ of $\Gamma$.

\subsection{}
\label{sec:U-P}

Let $\UU$ be the Zariski closure of $\Gamma_x'$ in $\G$. By the construction,
$\UU$ is a
$\Q$-subgroup with $\UU(\R) = U$. Let $\P$ be the normalizer of $\UU$ in $\G$.
This is a parabolic subgroup defined over $\Q$, with $\P(\R) = G_x$. 
We have a Levi decomposition $\P =
\UU \cdot \mathbf{L}$, where $\mathbf{L}$ is a reductive $\Q$-group. We denote
by $\ZZ$ the connected component of the center of $\mathbf{L}$. 
By the construction it is a torus defined over $\Q$.



\begin{prop}
        \label{prop:action-center}
        The $\Q$-torus $\ZZ$ is one-dimensional and split.
        It acts by conjugation on $\UU$ as follows (for $g \in \ZZ$, $b \in \UU$):
        \begin{align}
                g b g^{-1} = \lambda(g) b, 
                \label{eq:lambda}
        \end{align}
        where $\lambda$ is a nontrivial $\Q$-character of $\ZZ$.
\end{prop}
\begin{proof}
        The decomposition \eqref{eq:G_x} for $\P(\R)$ 
        shows that $\ZZ(\R) \cong \R^\times$, explicitly given by the
       group $A$ extended by the rotation $-I \in S$.
       This shows the existence of a nontrivial
       $\R$-isomorphism $\lambda$ of $\ZZ$ such that 
       \eqref{eq:lambda} holds for any $g \in \ZZ$ and $b \in \UU$.
       But by construction $\ZZ$ is a $\Q$-group that normalizes $\UU$, so that 
       for $g \in \ZZ(\Q)$ and $b \in \UU(\Q)$ we must have $\lambda(g) b \in
       \UU(\Q)$. This forces $\lambda(g) \in \Q^\times$. Since $\ZZ(\Q)$ is
       Zariski-dense in $\ZZ$ (see \cite[Cor.~18.3]{Bor91})
       we conclude that $\lambda$ is defined over $\Q$,
       and so $\ZZ$ is $\Q$-split.
\end{proof}

\begin{lemma}
        \label{lemma:homology-cusp-trivialized}
        Under the map induced by the inclusion, 
        $H_{n-1}(\UU(\Q))$ has trivial image in $H_{n-1}(\P(\Q))$.
\end{lemma}
\begin{proof}
        The conjugation induces a trivial action of $\P(\Q)$ 
        on its homology $H_{n-1}(\P(\Q))$ 
        (see \cite[Prop.~II.6.2]{Brown82}), so that the map $H_{n-1}(\UU(\Q)) \to
        H_{n-1}(\P(\Q))$ factors as follows:
        \begin{align}
                \begin{split}
                \xymatrix{
                        H_{n-1}(\UU(\Q)) \ar[rd] \ar[r] &
                        H_{n-1}(\UU(\Q))_{\P(\Q)}
                        \ar[d]\\
                        & H_{n-1}(\P(\Q)),
                }
                \end{split}
        \end{align}
        where $H_{n-1}(\UU(\Q))_{\P(\Q)}$ denotes the  module of co-invariants.
Since $\UU(\Q) \cong \Q^{n-1}$ we have $H_{n-1}(\UU(\Q)) = \Q$ (see
        \cite[Th.~V.6.4]{Brown82}), and the action of $\ZZ$ on
        $H_{n-1}(\UU(\Q))$ induced by conjugation is explicitly given by 
        \begin{align*}
                (g , u) \mapsto  \left(  \lambda(g)\right)^{n-1} \cdot u.
        \end{align*}
        It easily follows that the module $H_{n-1}(\UU(\Q))_{\P(\Q)}$ 
        is trivial. 
\end{proof}

\subsection{}
\label{sec:homol-Gamma_x}

For a cusp $x \in \Omega$ of a torsion-free $\Gamma$, we have that 
$H_{n-1}(\Gamma_x) \cong \Z$ (since $\Gamma_x\bs \R^{n-1}$ is a compact manifold).
Let us denote by $[\Gamma_x] \in H_{n-1}(\Gamma_x)$ a generator.

\begin{prop}
        \label{prop:Gamma_x-torsion}
        Let $\Gamma \subset \G(\Q)$ be a nonuniform quasi-arithmetic lattice.
        There exists an integer $N$ such that for any cusp $x$
        of $\Gamma$ the image of $[\Gamma_x]$ in $H_{n-1}(\P(\Q))$ is annihilated by
        $N$.
\end{prop}
\begin{proof}
        Let $\Gamma_x' \cong \Z^{n-1}$ be the maximal abelian subgroup of $\Gamma_x$. 
        The natural map $H_{n-1}(\Gamma_x') \to H_{n-1}(\Gamma_x)$ corresponds to $\Z \to
        \Z$ with $1 \mapsto s$, where $s = [\Gamma_x : \Gamma_x']$ is the index.
        Thus a generator of $H_{n-1}(\Gamma_x')$ is mapped to $s \cdot [\Gamma_x] \in
        H_{n-1}(\Gamma_x)$. But by construction $\Gamma_x' \subset \UU(\Q)$, and applying
        Lemma \ref{lemma:homology-cusp-trivialized} we see that $s \cdot [\Gamma_x]$ has zero
        image in $H_{n-1}(\P(\Q))$. Thus choosing $N$ as in Section \ref{sec:Bieberbach} the
        result follows.
\end{proof}

\section{Conclusion of the proof}
\label{sec:conclusion}

\subsection{}
Let $\Gamma \subset \G(k)^+$ be a torsion-free quasi-arithmetic lattice. 
The $\G(k)$-module $J\Omega$ fits into the following exact sequence:
\begin{align}
        0 \to J\Omega \to \Z \Omega \to \Z \to 0.
        \label{eq:J-ex-seq}
\end{align}
From this we obtain the commutative diagram with exact rows: 
\begin{align}
        \begin{split}
\xymatrix{
        H_n(\Gamma) \ar[d] \ar[r] & {H_{n-1}(\Gamma, J\Omega)} \ar[d]^{\alpha_*} \ar[r] \ar@{->}[rd]^\varphi &  H_{n-1}(\Gamma,
        \Z\Omega)  \ar[d] \\
        H_n(\G(k)) \ar[r]^-{\delta} & H_{n-1}(\G(k),J\Omega) \ar[r] & H_{n-1}(\G(k),
        \Z\Omega).
}
        \end{split}
        \label{eq:commut-diag}
\end{align}
The vertical maps are induced by the inclusion $\alpha: \Gamma \to \G(k)$.
Recall that the fundamental class $[\Gamma]$ corresponds to a generator of $H_n(\Gamma, \Omega)
= H_{n-1}(\Gamma, J\Omega) \cong \Z$. 
Consider the map $\varphi$ defined in \eqref{eq:commut-diag}.

\begin{prop}
        Let the integer $N$ be as in Proposition \ref{prop:Gamma_x-torsion}. 
        Then $\varphi(N \cdot [\Gamma]) = 0$. 
         \label{prop:phi-0}
\end{prop} 
\begin{proof}
 By Shapiro's lemma  the module $H_{n-1}(\Gamma, \Z\Omega)$ decomposes as a direct
 sum of modules $H_{n-1}(\Gamma_x)$, indexed by the set of $\Gamma$-orbits of points
 $x \in \Omega$.  But $H_{n-1}(\Gamma_x) = 0$ unless $x$ is a cusp, so that the sum is
 actually indexed by the quotient set $\Gamma\bs\CC$.

 Let us assume that $\G$ is isotropic (so that $k = \Q$).
 Since all cusps are conjugate by the action of $\G(\Q)$,
 we have by Shapiro's lemma:
 \begin{align*}
         H_{n-1}(\G(\Q), \Z \Omega) = H_{n-1}(\P(\Q)), 
 \end{align*}
 where $\P$ is constructed as in Section \ref{sec:U-P} for some cusp $x \in \Omega$.
 From Proposition \ref{prop:Gamma_x-torsion} we have that the image of $H_{n-1}(\Gamma,
 \Z \Omega) = \bigoplus_{x \in \Gamma\bs\CC} H_{n-1}(\Gamma_x)$
 in $H_{n-1}(\G(\Q), \Z\Omega)$ is annihilated by $N$. In particular, $N \cdot
 \varphi([\Gamma]) = 0$.
\end{proof}

\label{sec:end-of-proof}

\begin{proof}[Proof of Theorem \ref{thm:fund-class}]
        Let $L_0$ be the image of $H_n(\G(k))$ in $H_{n-1}(\G(k), \Omega)$
        under the connecting map $\delta$  (cf.\ \eqref{eq:commut-diag}).
        By Proposition \ref{prop:Borel-rk-1}, $L_0$ has rank one.
        It follows from the exactness of the second row in \eqref{eq:commut-diag}
        and Proposition 
        \ref{prop:phi-0} that $N \cdot \alpha_*([\Gamma]) \in L_0$ for 
        any torsion-free lattice $\Gamma \subset \G(k)^+$. Let $L_1$ be the
        image of $L_0$ in $H_n(G, \Omega)$, and denote by $L$ the submodule of
        $H_n(G, \Omega)$ generated by the elements $j_*([\Gamma])$ for
        torsion-free lattices $\Gamma \subset \G(k)^+$. Then $N \cdot L \subset
        L_1$, so that $\rk(L) = \rk(N \cdot L) = 1$ (note that this rank cannot
        be zero by Proposition \ref{prop:class-determ-vol}). 
\end{proof}

\bibliographystyle{amsplain}
\bibliography{vol-qa.bbl}

\end{document}